\documentclass[12pt]{amsart}
\usepackage{amssymb} 
\usepackage{amsthm}
\usepackage{mathtools}
\usepackage{mathrsfs} 
\usepackage{dsfont} 
\usepackage{bm}
\usepackage{bbm}  
\usepackage{cite}
\usepackage{enumerate}
\usepackage{units}
\usepackage{syntonly}
\usepackage{stdclsdv}
\usepackage{alltt}
\usepackage{fontenc}
\usepackage{ulem}
\usepackage{fancyhdr}
\usepackage{eqnarray}
\usepackage{calc}
\usepackage{footnote}
\numberwithin{equation}{subsection}
\theoremstyle{plain}
\newtheorem{theo}{Theorem}
\newtheorem*{theo*}{Theorem}

\theoremstyle{definition}
\newtheorem{lem}{Lemma}[section]

\newtheorem{prob}[lem]{Problem}
\newtheorem{prop}[lem]{Proposition}

\theoremstyle{remark}
\newtheorem*{ack}{Acknowledgements}
\newtheorem*{rem*}{Remark}
\newtheorem{rem}[lem]{Reamrk}

\newtheorem{cor}[lem]{Corollary}
\begin{document}

\title{Viscosity solutions, ends  and ideal boundaries}
\author{Xiaojun Cui }

\address{Xiaojun Cui\endgraf
Department of Mathematics, Nanjing University,
Nanjing 210093, Jiangsu Province,  People's
Republic of China.}
\email{xcui@nju.edu.cn}


\keywords{Busemann function; Horo-function;  Viscosity solution; $dl$-function; End; Ideal boundary}
\subjclass[2010]{49L25, 53C22, 53D25, 70H20}
\thanks{Supported by the National Natural Science Foundation of China
 (Grant 11271181), the Project Funded by the Priority Academic Program Development of Jiangsu Higher Education Institutions  (PAPD) and the Fundamental Research Funds for the Central Universities.}

 \abstract
On a smooth, non-compact, complete, boundaryless, connected Riemannian manifold $(M,g)$, there are three kinds of objects that have been studied extensively:

$\bullet$ Viscosity solutions to the Hamilton-Jacobi equation determined by the Riemannian metric;

$\bullet$ Ends introduced by Freudenthal and more general other remainders from compactification theory;

$\bullet$ Various kinds of ideal boundaries introduced by Gromov.

In this paper, we will present  some initial relationship among these three kinds of objects and some related topics are also considered.

\endabstract
\maketitle

\section*{Background, preliminaries and results}

This paper is a continuance to the papers \cite{CC}, \cite{C}.  Our aim is to understand the dynamics of minimal geodesics on a non-compact Riemannian manifold from the viewpoint of Aubry-Mather theory.

Let $M$ be a smooth, non-compact, complete, boundaryless (in the usual sense of point set topology), connected  Riemannian manifold with Riemanian metric $g$. Let $d$ be the distance on $M$ and $| \cdot|_g$ the norm on the tangent bundle $TM$ and/or the contangent bundle $T^*M$  induced by the Riemannian metric $g$. Let $\nabla$ be the gradient determined by the Riemannian metric $g$. Throughout this paper, all geodesic segments are always parameterized to be unit-speed.   By a ray \cite{Bu}, we mean a geodesic segment $\gamma:[0, +\infty) \rightarrow M$ such that $d(\gamma(t_1),\gamma(t_2))=|t_2-t_1|$ for any $t_1,t_2 \geq 0$. Throughout this paper, $| \cdot |$ means Euclidean norm. By definition, the Busemann function associated to a ray $\gamma$, is defined as
$$b_{\gamma}(x):=\lim_{t \rightarrow +\infty}[d(x, \gamma(t))-t].$$
Clearly, $b_{\gamma}$ is a Lipschitz function with Lipschitz constant 1, i.e.
$$|b_{\gamma}(x)-b_{\gamma}(y)| \leq d(x,y).$$
Moreover, in \cite{CC}, the authors showed that Busemann functions  are in fact locally semi-concave with linear modulus (for the definition of local semi-concavity of linear modulus, we refer to \cite{CS}, or \cite{CC}).

There are also other two kinds of functions, both introduced by Gromov, may be regarded as the generalizations of Busemann functions. They are defined as follows. Let ${x_n}$ be a sequence of points in $M$ such that $d(y,x_n) \rightarrow \infty$ for some fixed point $y$ (hence for any other fixed point in $M$) and  $$h(x):= \lim[d(x,x_n)-d(y,x_n)]$$ exists in the compact-open topology. Such a limit function will be called horo-function. More generally, let ${K_n}$ be a sequence of closed subsets in $M$ such that $d(y,K_n) \rightarrow \infty$ for some fixed point $y$ (hence for any other fixed point in $M$) and  $$h(x):= \lim[d(x,K_n)-d(y,K_n)]$$ exists in the compact-open topology. Such a limit function will be called $dl$ (distance-like)-function (here, the definition of $dl$-function is slightly different from the original definition of Gromov [\cite{Gr2}, Page 202], but it will not cause any confusion. See Remark 2.2 for further details).

As it is  explained explicitly in \cite{Mar} (see also \cite{Ba}), we could define at least three kinds of ideal boundaries \cite{Gr1}, \cite{Gr2} (all of them are equipped with the quotient compact-open topologies):

 $$M(\infty):= \{\text{Busemann functions}\}/\{\text{constant functions}\};$$

 $$M(\partial):= \{\text{horo-functions}\}/\{\text{constant functions}\};$$

$$M(\natural):= \{dl\text{-functions}\}/\{\text{constant functions}\}.$$

\begin{rem}
In the terminology of topology, here ideal boundary (not the usual one in the sense of point set topology and for short we just refer it as boundary in the rest of this paper) should be understood in this way: We could compactify $M$ (here $M$ is regarded as a non-compact, but locally compact, hemi-compact, Hausdorff, completely regular, $\cdots$, \footnote{In this paper, the non-compact complete Riemannian manifold $M$, regarded as a topological space, is at least Hausdoff, separable,  locally compact, hemi-compact (thus $\sigma$-compact), non-pseudo-compact, completely regular, Lindel\"{o}f (thus  realcompact),  perfectly normal.  For much more topological properties of non-compact complete Riemannian manifolds, we refer to [\cite{Ga2}, Theorem 2], [\cite{Ga3}, Theorem 1.1]. }  topological space) in various  ways.  A  compactification of $M$ means that there exist a compact Hausdorff \footnote{In this paper, we are only interested in the Hausdorff compactifications. Since $M$, as a complete Riemannian manifold,  is completely regular, Hausdorhoff compactifications do always exist [e.g. \cite{Ch}].} topological space, say $\bar M$, and a topological embedding $i$ of $M$ into $\bar M$ such that $i(M)$ is a dense subset of $\bar M$. Then we call $\bar M \setminus i(M)$  to be a boundary ($\bar M \setminus i(M)$ is also called to be a remainder by topologists working in the field of compactification theory) of $M$. If the  embedding $i$ we considered here  is defined in the usual way  by $i(x)=\frac{d(x, \cdot)}{\{\text{ constant functions}\}}$, then only $M(\partial)$ deserves the terminology  ``boundary"  and $M(\infty)$ is only a part of the boundary $M(\partial)$.  By abuse of notations, here we insist on  calling  $M(\infty)$ and $M(\natural)$ to be  boundaries. We will see that calling  $M(\natural)$ a boundary is reasonable in the light of Theorem 2 once the embedding is chosen suitably.
\end{rem}


By definitions, we have $M(\infty) \subseteq M(\partial) \subseteq M(\natural)$.  It has been realized \cite{AGW}, \cite{Co}, \cite{FM}, \cite{IM} that the set $M(\infty)$ is a good analogue of the set of  static classes of Aubry sets  in Aubry-Mather theory for positively definite Lagrangian systems (for details for Aubry-Mather theory, we refer to \cite{Mat1}, \cite{Mat2}, \cite{Fa}).  In \cite{CC} the authors  began to study the geometric property of the Riemannian metric from this viewpoint.  More precisely  [\cite{CC}, Theorem 1, Corollary 7.2] showed that all  $dl$-functions (including horo-functions and Busemann functuions) are viscosity solutions with respect to the Hamilton-Jacobi equation
$$|\nabla u|_g=1. \,\,\,\,\, \,\,\,\,\, \hfill{(*)}$$
A natural inverse problem is

\begin{prob}
Whether any viscosity solution to the Hamilton-Jacobi equation $(*)$
 must be a $dl$-function?
\end{prob}

In this paper, we will show that the answer to this problem is (almost) yes! Precisely, we have the following result.

\begin{theo}
Up to a constant, a function $f$ is a viscosity solution with respect to the Hamilton-Jacobi equation
$$|\nabla u|_g=1$$
if and only $f$ is a $dl$-function. Thus,
$$M(\natural)=\{\text{viscosity solutions}\}/\{\text{constant functions}\}.$$
\end{theo}

By equipping compatible (with respect to the manifold topology) proximities or  totally bounded uniforms, one could obtain various kinds  compactifications and thus get various kinds of  boundaries \cite{Wi}, \cite{Jo} from the viewpoint of general topology. Among them, the following three are particularly important: Alexandroff compactification (i.e. one-point compcatification), Freudenthal end-compactification and  Stone-\v {C}ech compactification.  On the set of compactifications, a very natural partial order $``\leq "$ is  well defined. Among the set of  compactifications, the Alexandroff compactification is the smallest  compactification and the Stone-\v{C}ech compactificaton is the  largest one with respect to the partial order $``\leq "$. We say that a compactification $C_1$ \footnote{Here, and somewhere else, we do not specify the embedding $i$ and identify $i(M)$ with $M$ whenever we believe that no confusion would be caused.} is smaller than a compactification $C_2$ if $C_1 \leq C_2$ and that $C_1$ is strictly smaller than $C_2$ if $C_1 \leq C_2$ but $C_2 \nleq C_1$. We say that $C_1$ is equivalent to $C_2$ ( $C_1 \simeq C_2 $ ) if $C_1 \leq C_2$ and $C_2 \leq C_1$.  That $C_1$ is equivalent to $C_2$ would imply that $C_1$ is homemorphic  to $C_2$  and the converse is not true in general. The set of equivalence classes of compactifications is a complete lattice in our case, since $M$ is locally compact. Also note that since $M$ is completely regular, the set of equivalence classes of compactifications corresponds to the set of closed separating \footnote{We say a subset $\mathcal{S}$ of $C(M, \mathbb{R})$ is separating if for any closed subset $K$ of $M$, and any point $x \notin K$, there exists $f \in \mathcal{S}$ such that $f(x) \notin \overline{f(K)}$, the closure of $f(K)$.} subalgebra (containing the constant function) of real-valued function [e.g. \cite{HS}, Page 71].  For more information on compactification theory,  we refer to \cite{Ch}, \cite{GJ}, \cite{Wa}, \cite{HS}.

Recall that a compact, connected, locally connected, metric space is refereed to as a Peano space.
\begin{theo}
$M(\natural)$ is a Peano space and consequently $M(\natural)$ is a boundary (i.e. remainder) of $M$. Moreover, any compactification  with $M(\natural)$ as boundary (i.e. remainder) is strictly smaller than the Stone-\v{C}ech compactification.
\end{theo}

In this paper we use $\#$ to denote the cardinality of a set. For the initial relations among the three kinds of (ideal) boundaries, we have
\begin{theo}
1). If $\#(M(\infty))=1$, then $\#(M(\natural))=1$.

2). If  $\#(M(\infty))=1$, then any viscosity solution to the Hamilton-Jacobi equation
$$| \nabla u|_g=1$$
attains its maximum at a compact subset of $M$.  Consequently, this Hamilton-Jacobi equation admits no $C^1$ solutions.
\end{theo}

For a noncompact topological space, following Freudenthal \cite{Fr} we could define its (topological) ends.  In our special setting of noncompact, complete Riemannian manifold, we could define  them by rays [\cite{Ab}, III 2].  An equivalence class of cofinal (here two rays $\gamma_1$ and $\gamma_2$ are called to be cofinal if  for any compact subset $K$, there exists $t_K >0$ such that $\gamma_1(t_1)$ and $\gamma_2(t_2)$ lie in the same connected component of $M \setminus K$ for all $t_1, t_2 \geq t_K$ ) rays is called an end of $M$.  Let $\mathscr{E}(M)$ be the set of ends, equipped with the natural topology.  It is known that $\mathscr{E}(M)$ is a totally disconnected Hausdorff space, and it is exactly a kind of boundary (i.e. remainder) with respect to the Freudenhal end compactification  in the sense of Remark 0.1. For more details on ends theory, we refer to \cite{Fr}, \cite{Ho},\cite{Ab}, \cite{Gu}, \cite{Ei}.

\begin{rem}
Different from $\mathscr{E}(M)$, the ideal boundaries $M(\infty)$, $M(\partial)$ and $M(\natural)$ depend on the Riemannian metric.
\end{rem}

For a ray, it could represent either an element of the (metric) ideal boundary $M( \infty)$ or an element of the (topological) $\mathscr{E}(M)$, thus connect these two objects. To state results along this line, we first introduce some notations.

Given a  ray $\gamma$, we could define its coray as follows.  A ray $\gamma^{\prime}:[0, \infty) \rightarrow M$ is called to be a coray to $\gamma$ if there exist a sequence  $x_k \rightarrow \gamma^{\prime}(0)$, a sequence  $t_k \rightarrow \infty$  and a sequence of minimal geodesic segments $\gamma_k :[0, d(x_k, \gamma(t_k)] \rightarrow M$ connecting $x_k$ and $\gamma(t_k)$ such that $\gamma_k$ converge to $\gamma^{\prime}$ uniformly on any compact interval of $[0, \infty)$.  Now we have

\begin{theo}
1) [\cite{C}, Theorem 7].  For any ray $\gamma$, if $\gamma^{\prime}$ is a coray to $\gamma$, then $\gamma$ and $\gamma^{\prime}$ are cofinal.

2). If $\#(\mathscr{E}(M)) \geq 3$, then for any Riemannian metric $g$ on $M$, the assocaited Hamilton-Jacobi equation
$$| \nabla u|_g=1$$
admits no $C^1$ solutions.
\end{theo}

The significance  of   results in the present paper is that it builds some  interesting  connections among  these three kinds of objects: viscosity solutions \cite{Li}, \cite{CS};  ideal boundaries $M(\infty)$, $M(\partial)$ and  $M(\natural)$ \cite{Gr1}, \cite{Gr2}; the set of (topological) ends $\mathscr{E}(M)$  \cite{Fr} or more general remainders in compactification theory.
  It looks  at the first glimpse that they  belong somehow to  different fields. By the results here, to understand the global property of viscosity solutions of Hamilton-Jacobi equation $(*)$, it is necessary to study deeply the structures of $M(\infty), M(\partial),M(\natural), \mathscr{E}(M)$ and the relations among them. Also it would be interesting if one could relate deeply the structure of these sets to the geometric properties of the Riemannian metric (For Hardmard manifolds,  Manifolds of negative curve  or Riemannian manifolds which are convex at infinity, progress are fruitful (e.g. \cite{Ba},  \cite{DPS}, \cite{Yi}); for general cases, very little is known) or the dynamics of the geodesic flow. The results in the present paper may be regarded as initial progress and we hope to come back to this issue in the future.

We will divide the proof of Theorem 1 into two parts and leave it to the following two sections. In section 3, we provide some simple applications of Theorem 1. In section 4, we will analysis the topological structure of $M(\natural)$ and prove Theorem 2. In section 5,  we will consider some consequences when $M(\infty)$ is a singleton and prove Theorem 3. In section 6, we consider  some initial relations among  ends, ideal boundaries and viscosity solutions and thus prove Theorem 4.

\section{$dl$-functions are viscosity solutions}

In \cite{CC}, that Busemann functions are viscosity solutions is proved in details, but for $dl$-functions (or horo-functions), the result is stated as a corollary [\cite{CC}, Corollary 7.2] without details. Although the proof is almost the same to the case of Busemann functions, we still give a relatively detailed sketch as follows, for the sake of completeness.

Assume that $f$ is a $dl$-function, i.e.  there exists a sequence of closed subsets $K_n$, with $d(x_0, K_n) \rightarrow \infty$ for some fixed point $x_0 \in M$ such that
$$[d(\cdot, K_n)-d(x_0,K_n)] \rightarrow f$$
in the compact-open topology. Then, by the following steps, we could prove that $f$ is a viscosity solution and thus is locally semi-concave with linear modulus.

\emph{Step 1:} By [\cite{MM}, Theorem 3.1, Proposition 3.4], for any nonempty closed subset $K_n$, $d(x,K_n)-d(x_0,K_n)$ is a viscosity solution with respect to the Hamilton-Jacobi equation $(*)$. 

\emph{Step 2:} Since $d(x_0, K_n) \rightarrow \infty$ and $[d(\cdot, K_n)-d(x_0,K_n)] \rightarrow f$ in the compact-open topology (in our case, the compact-open  topology coincides with the topology of uniform convergence on compacta), by the stability of viscosity solutions [\cite{CS}, Theorem 5.2.5], $f$ is a global viscosity solution to the Hamilton-Jacobi equation $(*)$.

\emph{Step 3:} Since Hamilton-Jacobi equations $(*)$ and
$$|du|^2_g=1 \,\,\,\,\,\, (**)$$ admit the same set of viscosity solutions, we may also regard $f$ as  a viscosity solution to the Hamilton-Jacobi equation $(**)$. Since $(**)$ is induced by the locally uniformly convex Hamiltonian $H(x,p):=|p|^2_g$,
$f$, as a viscosity solution of a locally uniformly convex Hamilton-Jacobi equation,  must be locally semi-concave with linear modulus [\cite{CS}, Theorem 5.3.6].

\section{ Viscosity solutions are $dl$-functions up to a constant}

Given a viscosity solution $f$ to the Hamilton-Jacobi equation $(*)$, we will show that $f$ is a $dl$-function
for a suitable sequence of closed subsets $K_n$ up to a constant. The crucial point is to choose a suitable sequence of closed subsets $K_n$.

For any $a \in \mathbb{R}$, let $H_a:=\{x: f(x)=a\}$.  Choose any fixed point $x_0 \in M$ and assume that $f(x_0)=a_0 \in \mathbb{R}.$  Let $K_n:=H_{-n}$, we will show that the sequence of closed subsets  $K_n$ will be the one we are looking for.

Firstly, $K_n$ is a closed subset for any $n \in \mathbb{Z}^+$, since $f$ is a continuous (in fact, Lipschitz) function.  Also note that $K_n$ is non-empty for all suitable large $n$.  In the following we will show that
$$[d(\cdot, K_n)-d(x_0,K_n)] \longrightarrow f  $$ in the compact-open topology.

Now we will prove that for any two real numbers $a_1 > a_2 \in \mathbb{R}$, for any point $x \in H_{a_1}$, $y \in H_{a_2}$, $d(x,H_{a_2})=a_1-a_2$.   Choosing a reachable (unit) vector $v \in TM_x$ (i.e. there exists a sequence $x_i \rightarrow x$ such that $f$ is differentiable at $x_i$ and $\nabla f(x_i) \rightarrow v$), there exists a unique minimal geodesic segment $\gamma:(-\infty, 0]$ with $ \gamma(0)=x$ and $\dot{\gamma}(0)=v$ and $f(\gamma(t))-f(x)=t$, since for the  Hamilton-Jacobi equation whose Hamiltonian is locally uniformly convex, viscosity solutions coincide with variational (minmax) solutions \cite{Zh1}, \cite{Zh2}.   Thus, there exists a unique $t_0$ such that $f(\gamma(t_0))=a_2$. It is easy to see that $t_0=a_2-a_1$ and
$$d(x, \gamma(t_0))=length(\gamma|[t_0,0])=|t_0|=-t_0.$$ Since $\gamma(t_0) \in H_{a_2}$, we get $d(x, H_{a_2})\leq -t_0=a_1-a_2.$  If $a_3:=d(x, H_{a_2})<a_1-a_2$,  there exists a minimal geodesic segment $\xi:[-a_3,0] \rightarrow M$  with $\xi(-a_3) \in H_{a_2}$ and $\xi(0)=x \in H_{a_1}$.  Then $f \circ \xi:[-a_3,0] \rightarrow \mathbb{R} $ is still a Lipschitz function,  thus  differentiable with respect to the 1-dimensional Lebesgue measure.   So  we get
\begin{eqnarray*}&&a_1-a_2\\
&=&f(\xi(0))-f(\xi(-a_3))\\
&=&\int^0_{-a_3} g(\nabla f, \dot{\xi})dt \\
&\leq & length (\xi|_{[-a_3,0]}) \,\,\,\,\,\,\,\,\,\,\,   (\text{ Since $|\nabla f|_g \leq 1$})\\
&=& a_3\\
&<& a_1-a_2.
\end{eqnarray*}
This contradiction proves $d(x, H_{a_2})=a_1-a_2$.

By the discussions above, for any compact subset  $S$, there exists a constant $n_S >0$ such that for any $n > n_S$,
$$-n < min_{x \in S}f(x),\,\,\,\,\,\,\,\, -n <f(x_0) .$$

Then for any $x \in S$ and any $n >n_S $,
$$d(x, K_{n})=d(x, H_{-n})=f(x)+n$$
and
$$d(x_0, K_{n})=d(x_0, H_{-n})=f(x_0)+n.$$

It means that
$$d(x,K_n)-d(x_0,K_n)=f(x)-f(x_0)$$ for any $x \in S$ and any $n > n_S$.   This is to say, $$[d(x, K_n)-d(x_0,K_n)] \longrightarrow f(x)-f(x_0)$$ in the compact-open topology.

\begin{rem}
By the proof of Theorem 1, we could obtain this fact: If $f$ is a viscosity solution and there exists a point $x_0$ such that $f(x_0)=0$, then $f$ itself is a $dl$-function. In other words, it is not necessary to add a constant in this case.
\end{rem}
\begin{rem}
In [\cite{Gr2}, Page 202], $dl$-function is defined in a slightly different form: a  function $f$ is a $dl$-function if $$f(x)=t+d(x, f^{-1}(-\infty,t])$$ for all $t \in \mathbb{R}$ and those $x \in M$ where $f(x) \geq t$. By the procedure of the proof of Theorem 1 in section 2, this definition coincides with the one we used up to a constant, thus we could use either of them to define the ideal boundary $M(\natural)$.  If we prefer the original one, Theorem 1 could be restated as: $f$ is a viscosity solution of Hamilton-Jacobi equation $(*)$ if and only if $f$ is a $dl$-function.  Also, the definition of $dl$-function is also slightly different from the one in \cite{Mar}, where an element in $M(\natural)$ is called to be a $dl$-function.
\end{rem}

Combining the contents of section 1 and section 2, Theorem 1 is proved.

\section{Some applications}
By Theorem 1,  $M(\natural)$ could be redefined as

$$M(\natural)=\{\text{viscosity solutions}\}/\{\text{constant functions}\}.$$
Thus, $M(\natural)$  should inherit some properties from  viscosity solutions. Here we collect some well-known  ones, which will be useful later.

\begin{cor}[{[\cite{Mar}, Lemma 4.5]}]
If $f_1,f_2$ are two $dl$-functions, then $\min (f_1,f_2)$ is still a $dl$-function up to a constant.
\end{cor}

The proof in \cite{Mar} is totally different from the one presented below.
\begin{proof}
If $f_1,f_2$ are two $dl$-functions, then they are viscosity solutions and locally semi-concave with linear modulus.

First we will show that $\min(f_1,f_2)$ is  still a viscosity solution to the Hamilton-Jacobi equation $(*)$ (or equivalently $(**)$).  By [\cite{CS}, Proposition 1.1.3],  it is easy to obtain that $\min(f_1,f_2)$ is still a locally semi-concave function with linear modulus. Thus, by [\cite{CS}, Proposition 5.3.1], we only need to show that Hamilton-Jacobi equation $(**)$ is satisfied at all points of differentiable points of $\min(f_1,f_2)$. Let $M_1:=\{x: f_1(x) \neq f_2(x)\}$ and $M_2:=\{x: f_1(x)=f_2(x)\}$. We denote the interior of $M_2$ by $int(M_2)$.  Let $U:=M_1 \cup int(M_2)$,  then  $\min(f_1,f_2)$ satisfies the Hamilton-Jacobi equation $(**)$ at  its differentiable points  in $U$. Note that $U$ is an open and dense subset, by [\cite{CS}, Proposition 3.3.4 (a)], together with the continuity of $|\cdot|_g$,  $\min(f_1,f_2)$ satisfies  the Hamilton-Jacobi equation $(**)$ at any differentiable point. So far, we know that $\min(f_1,f_2)$ is really a viscosity solution to the Hamilton-Jacobi equation $(**)$ (or equivalently $(*)$). By Theorem 1, this means that $\min(f_1,f_2)$ is a $dl$-function up to a constant.
\end{proof}
 
We  could state Corollary 3.1 alternatively as 

\begin{cor}
If $f_1, f_2$ are two viscosity solutions, then $\min{\{f_1, f_2\}}$ is a  viscosity solution too.
\end{cor}
\begin{cor}
$M(\natural)$ is compact with respect to the quotient compact-open topology.
\end{cor}

\begin{proof}
Fix a point $x_0 \in M$ and we represent  elements of $M(\natural)$ by  $dl$-functions $f$ with $f(x_0)=0$. In other words, we identify  $M(\natural)$ with
$$\mathcal{F}:=\{f: f \text{ are $dl$- functions with $f(x_0)=0$}\}.\footnote{$\mathcal{F}$, equipped  with the the compact-open topology, and $M(\natural)$, equipped with the quotient compact-open topology, are homemorphic. In fact, the map $\mathcal{F} \ni f \mapsto f /\sim $ is the homemorphism.}$$
For any sequence $f_n \in \mathcal{F}$, since $f_n$ is uniformly  Lipschitz (with Lipschitz constant 1, thus equi-continuous) and  uniformly bounded on any compact subset, together with the fact that $M$ is hemi-compact [\cite{Ga2}, Theorem 2] with respect to the manifold topology (or equivalently the topology induced by the distance $d$), there exists a subsequence $f_{n_i}$ such that $f_{n_i}$ converge to a continuous function $f$ in the compact-open topology, by Arzela-Ascoli theorem. Further by the stability of viscosity solutions [\cite{CS}, Theorem 5.2.5], $f$  itself is a viscosity solution and $f(x_0)=0$.   Namely, $f \in \mathcal{F}$. So far we have proved that $\mathcal{F}$ (or equivalently $M(\natural)$) is sequentially compact.

By [\cite{Mar}, Theorem 4.6], $M(\natural)$ is metrizable \footnote{Since $M$ is hemi-compact, $\mathcal{C}(M)$, the set of continuous functions on $M$ with the compact-open topology, is metrizable (e.g. [\cite{Wi}, 43G]). The metric  $\rho$ could be defined by 
$$\rho(u,v)=\sum^{\infty}_{n=1} \rho_n(u,v),$$
where $K_n=\overline{B_n(x_0)}$, the closed metric ball centered at some fixed point $x_0$ with radius $n$; $\rho_n(u,v)=\min \{\frac{1}{2^n}, \sup_{x \in K_n}|u(x)-v(x)|\}$.  Thus the set $\mathcal{V}$ of viscosity solutions to the Hamilton-Jacobi equation (*),  as a subspace of $\mathcal{C}(M)$, is metrizable. We consider $\mathbb{R}$ as an additive group, acting on $\mathcal{V}$ by $t u:=t+u$.  Then for each $t \in \mathbb{R}$, $t$ is an isometry and moreover the orbits of the action are closed. By [\cite{CDL}, Theorem 2.1], $M(\natural)$, as a quotient space where the equivalene relation (denoted by $\sim$) is induced by the $\mathbb{R}$-action,  is metizable  by quotient metric $\rho_{\sim}$, here $\rho_{\sim}$ is defined by 
$$\rho_{\sim}(u/\sim, v/\sim)=\inf_{t,s \in \mathbb{R}} \rho(u+t, v+s).$$ Be careful that in general case on a quotient space of  a metric space only quotient pseudo-metric is well defined.} . It is well known in a metric space, sequential compactness and compactness are equivalent (e.g.  [\cite{Ja}, Page 84, Proposition 3]). Thus, $M(\natural)$ is  compact with respect to the quotient compact-open topology.

\end{proof}

Let the horo-function compactification $$\bar{M}(\partial):=closure \frac{\{d(\cdot, x): x \in M\}}{ \{ \text{constant functions} \} },$$ then it is easy to see that $M(\partial)$ is the topological boundary  of the set $\bar{M}(\partial)$, here  ``closure" and ``boundary"  are considered under the quotient compact-open topology. It is known that $\bar{M}(\partial)$ is a compact, metrizable (particularly, Hausdorff) space. Hence  $M(\partial)$ is a closed subset of $\bar{M}(\partial)$. By this fact,   we could get the following well known result.
\begin{cor}
$M(\partial)$ is compact with respect  to the quotient compact-open topology.
\end{cor}



\section{Proof of Theorem 2}
For the simiplicity of notations,  in this section we use $\sim$ (as in  footnote 6) to denote the equivalence relation  where two continuous functions are equivalent if they differ a constant. Now we will prove that $M(\natural)$ is a Peano space and thus $M(\natural)$ could also be regarded as a boundary (i.e. remainder). We will divide the proof to the following steps.

$\bullet$ \textsl{Compactness.} The compactness of $M(\natural)$ is proved in Corollary 3.3.

$\bullet$ \textsl{Connectedness.}  For any two viscosity solutions $u,v$, let $$f_t:=\min \{u+t,v\}.$$
By Corollary 3.2, for each $t \in \mathbb{R}$, $f_t$ is  a viscosity solution. Thus the map
$$t \rightarrow f_t$$
is a continuous map from $\mathbb{R}$ to $\mathcal{V}$, the set (equipped with the compact-open topology) of viscosity solutions to the Hamilton-Jacobi equation $(*)$. It is also easy to see that $f_t/ \sim  \longrightarrow u/ \sim$ as $t \longrightarrow -\infty$ and $f_t/ \sim \longrightarrow v/ \sim$ as $t \longrightarrow \infty$. This shows that $\mathcal{V}$, and thus $M(\natural)$, is connected.

$\bullet$ \textsl{Local connectedness.} By [\cite{Wi}, 27.16 Theorem], we only need to show that $M(\natural)$ is connected im Kleinen  at every point $u/ \sim \in M(\natural)$. In other words, we have to show that for any $u/\sim \in M(\natural)$ and for any neighborhood $\mathcal{N}$ of $u/\sim$ 
in $ M(\natural)$, there exists a connected neighborhood $\mathcal{M}$ of $u/\sim$ with $ \mathcal{M}  \subset \mathcal{N}$. Firstly, we choose $\epsilon >0$ small enough such that the metric \footnote{Recall that $M(\natural)$ is metrizable by the quotient metric $\rho_{\sim}$ introduced in footnote 6.} ball $B_{\epsilon}(u/\sim)$ is contained in $\mathcal{N}$. Now we construct the connected neigiborhood of $u/\sim$ as following. 

We fix a viscosity solution $u$ as a representation of the class $u/\sim$. For any  $v/\sim \in B_{\frac{\epsilon}{4}}(u/\sim)$,   we fix  a representation of $v/\sim$,  still denoted by $v$,  such that 
$\rho(u,v) <\frac{\epsilon}{4}$. Such a representation does exist  since  the $\mathbb{R}$-action, which induces the quotient equivalence relation, is an isometry (see footnote 6). For such element  $v$, let
$$f_t(u,v):=  \min\{u+t,v \} .$$
It is easy to see that 
$$\rho_{\sim}(f_t(u,v)/\sim, u/\sim ) \leq  \rho(f_t(u,v), u+t) \leq \rho(u,v)$$ for any $t \leq 0$ and  
$$\rho_{\sim}(f_t(u,v)/\sim, v/\sim) \leq  \rho(f_t(u,v), v) \leq  \rho (u,v)$$ for any $t \geq 0$.  Consequently $\rho_{\sim}(f_t (u,v) / \sim ,u/\sim )  <\epsilon$ for any $t \in \mathbb{R}$.
Now we define
$$\mathcal{M}:=\{\cup_{v/\sim  \in B_{\frac{\epsilon}{4}} (u/ \sim), t \in \mathbb{R} } f_t(u,v)\} \cup \{B_{\frac{\epsilon}{4}} (u/ \sim)\},$$
here $v$ is a representation of $v/\sim$ as explained at the begining of this in paragraph. Then we have $\mathcal{M} \subset B_{\epsilon}(u/\sim) \subseteq \mathcal{N}$. By the definition, $\mathcal{M}$ is connected and thus $M(\natural)$ is connected im Kleinen at $u/\sim$. Hence, $M(\natural)$ is locally connected.

$\bullet$ \textsl{Metrizability. } As we said in the proof of Corollary 3.3, the metrizability of $M(\natural)$ is proved in [\cite{Mar}, Theorem 4.6].  See also footnote  3.

By the result [\cite{Mag}, Corollary (2.3)],  any peano space could be  a boundary (i.e. remainder) of  any locally compact, non-compact, metric space.  Thus, $M(\natural)$ is a boundary (i.e. remainder) of $M$. 

Since $M$ is  connected, locally compact but not pseudo-compact, the Stone-\v{C}ech compactifiaction $\beta(M)$ is connected but not locally connected [\cite{HI}, 2.5 Corollary], [\cite{Wa},   9.3.Theorem ]. For the remainder  $\beta(M) \setminus M$,  the realcompactness of $M$ implies that it is not connected im Kleinen at any point [\cite{Wo}, Theorem 5]. Consequently, $\beta(M) \setminus M$ is not locally connected at any point.  On the other hand, since $M(\natural)$ is locally connected,   $M(\natural)$ is not homoemorphic to $\beta(M)\setminus M$, and thus  any  compactification with $M(\natural)$ as boundary (i.e. remainder) is not equivalent to the Stone-\v{C}ech compactification. Since Stone-\v{C}ech compactification is the largest compactification, any compactification with $M(\natural)$ as boundary is strictly smaller than the Stone-\v{C}ech compactification.

\section{the case $M(\infty)$ is a singleton}

 In weak KAM theory \cite{Fa}, it is well known that if there is only one static class for some Aubry set, then the associated Hamilton-Jacobi equation has only one viscosity solution up to a constant. Here we also provide an analogous  property (i.e. Theorem 3) in our setting.
We restate Theorem 3. 1) as

\begin{prop}
If $M(\infty)$ is a singleton, then $M(\natural)$ is a singleton as well.
\end{prop}
\begin{proof}
For  any fixed ray $\gamma$,  we will prove that for any viscosity solution $f$, $f=b_{\gamma}$ up to a constant.  Let $D$ be the set on which both $f$ and $b_{\gamma}$ are differentiable. Clearly, $D$ is of full measure with respect to the Lebesgue measure. Since $f$ is a viscosity solution, for any $x \in D$, there exists a unique ray $\gamma_x:[0, \infty) \rightarrow M$ such that $\gamma_x(0)=x$,  $-\dot{\gamma}_x(0)=\nabla f(x)$ and $f(\gamma_x(t_2))-f(\gamma_x(t_1))=t_1-t_2$ for any $t_1,t_2 \in \mathbb{R}$. Since $M(\infty)$ consists of only one point, $\gamma_x$ must be a coray to $\gamma$. Since $b_{\gamma}$ is differentiable at $x$, $\gamma_x$ is the only coray to $\gamma$. Thus, $\nabla f(x)=\nabla b_{\gamma}(x)=-\dot{\gamma}_x(0)$. Thus we get $\nabla(f-b_{\gamma})=0$ on $D$. Since $D$ is of full measure and $f-b_{\gamma}$ is a Lipschitz function, applying Fubuni's theorem we get $f-b_{\gamma}$ is a constant. So $M(\natural)$ is also a singleton.
\end{proof}

We restate Theorem 3. 2) as

\begin{prop}
If $M(\infty)$ is a singleton, then for any viscosity solution $f$, the maximum of $f$ is obtained at a bounded closed subset. In other words,   $f$ attains its maximum at a compact subset.
\end{prop}
\begin{proof}
By Proposition 5.1, we only need to show that the Proposition 5.2 is true for some Busemann function.
Fix a ray $\gamma$, we first prove that $b_{\gamma}$ is bounded above. Otherwise, there exists a sequence of points $x_n (n \geq 0)$ such that $b_{\gamma}(x_n) \rightarrow \infty$. Let $\gamma_n:[0, d(x_0,x_n)] \rightarrow M$ be a minimal geodesic segment connecting $x_0$ and $x_n$, then  by taking a subsequence if necessary, $\gamma_n$ $(n \geq 1)$ will convergence uniformly on any conpact time interval to a ray emanating from $x_0$. We denote this ray by $\gamma^{\prime}$. Clearly $\gamma^{\prime}$ and $\gamma$ represent two distinct elements in $M(\infty)$ (just recalling that if $\gamma^{\prime}$ and $\gamma$ represent the same element in $M(\infty)$, then $b_{\gamma}(\gamma(t_1))-b_{\gamma}(\gamma^{\prime}(t_2))=t_2-t_1$ for any $t_1,t_2 \geq 0$) and it contradicts the assumption that $M(\infty)$ is a singleton.

Now we prove that $\max{b_{\gamma}}$ exists.  Otherwise, there exists an undounded sequence of $x_n$ $(n \geq 0)$ such that $b_{\gamma}(x_n) \rightarrow \sup{b_{\gamma}}$. By the same discussion as in the previous  paragraph, we could get a ray $\gamma^{\prime}$, such that $\limsup_{t \rightarrow \infty}b_{\gamma}(\gamma^{\prime}(t))=\sup{b_{\gamma}}$. It implies  that $\gamma^{\prime}$ and $\gamma$ represent two distinct elements in $M(\infty)$ and thus we get a contradiction too.

Now we show that $\{x:b_{\gamma}(x)=\max{b_{\gamma}}\}$ lies in a bounded subset of $M$.  Otherwise, there exists an undounded sequence of $x_n$ $(n \geq 0)$ such that $b_{\gamma}(x_n) = \max{b_{\gamma}}$. By the same discussion as in the previous two paragraphs, we could get a ray $\gamma^{\prime}$, such that $b_{\gamma}(\gamma^{\prime}(t_n))=\max{b_{\gamma}}$  for an  unbounded sequence  $t_n$. It will also imply that $\gamma^{\prime}$ and $\gamma$ represent two distinct elements in $M(\infty)$ and thus we get a contradiction as well.

Since $(M,g)$ is assumed to be complete, by Hopf-Rinow theorem (e.g. \cite{Pe}, Page 137, Theorem 16]), M satisfies the Heine-Borel property, i.e. every bounded closed subset is compact. Thus,  $f$ attains its maximum at a compact subset.

\end{proof}

If $M(\infty)$ is a singleton, then any viscosity solution  attains its maximum at a compact subset, and of course is non-differentiable at the maximum points. In other words, Hamilton-Jacobi equation $(*)$ does not admit $C^1$ solutions.
So far Theorem 3 is proved.

\begin{cor}
If the Hamilton-Jacobi equation $(*)$ admits a $C^1$ solution, then $M(\infty)$ contains at least two elements. More generally, if $M$ admits a line (this is the case, for example, $\#(\mathscr{E}(M)) \geq 2$), then $M(\infty)$ is not a singleton; recall that by definition a line $\gamma: \mathbb{R} \rightarrow M$ is a geodesic such that $d(\gamma(t_1), \gamma(t_2))=|t_1-t_2|$ for ant $t_1,t_2 \in \mathbb{R}$.
\end{cor}

\section{Ends and ideal boundary}

 Since any two distinct ends can be connected by a line [\cite{Ab}, III 2.3], we easily get

\begin{cor}
If $M(\infty)$ is a singleton, then $\mathscr{E}(M)$ must be a singleton.
\end{cor}

  For corays, we have the following proposition. 
\begin{prop}[\cite{C}, Theorem 7]
For any ray $\gamma$, all corays to $\gamma$ are cofinal to $\gamma$.
\end{prop}

Theorem 4. 1) is just a restatement of Proposition 6.2.

 By Proposition 6.2, we easily get
\begin{cor}
$\#(\mathscr{E}(M)) \leq \#(M(\infty))$.
\end{cor}

If  $M=\mathbb{R}$, both $ \mathscr{E}(M)$ and $M(\infty)$ always contain exactly two elements. For higher dimensional cases,  $\mathbb{R}^n$ ($n \geq 2$) is of one-end, i.e. $\#(\mathscr{E}(M))=1$ ($\mathscr{E}(M)$ is a topological notion and independent of the Riemannian metric). We would like to pose

\begin{prob}
On $\mathbb{R}^n (n \geq 2)$, characterize all Riemannian metrics such that  the associated ideal boundary $M(\infty)$ is a singleton.
\end{prob}

More generally, we pose

\begin{prob}
On any non-compact, boundaryless, connected, paracompact  manifold $M$, is there a complete Riemannian metric $g$ such that the associated ideal boundary $M(\infty)= \mathscr{E}(M)$ (i.e. two rays $\gamma$ and $\gamma^{\prime}$ are cofinal if and only if $b_{\gamma}-b_{\gamma^{\prime}}=const.$)?
\end{prob}
Now we restate Theorem 4. 2) as
\begin{prop}
 If $\#(\mathscr{E}(M)) \geq 3$, then for any Riemannian metric $g$ on $M$, the associated Hamilton-Jacobi equation
$$| \nabla u|_g=1$$
admits no $C^1$ solutions.
\end{prop}
\begin{proof}
Otherwise, there exists a Riemannian metric $g$ such that the associated Hamilton-Jaocbi equation
$$| \nabla u|_g=1$$
admits a $C^1$  solution, say $f$. So  $-f$ is also a $C^1$ solution to the Hamilton-Jacobi equation $(*)$. So both $f$ and $-f$ are locally semi-concave with linear modulus, and thus $f$ is locally $C^{1,1}$ [\cite{CS}, Corollary 3.3.8]. So we obtian that  $\nabla f$ is a locally Lipschitz vector field and consequently existence and uniqueness property of solutions of ODE holds.    Clearly, the integral curves of $\nabla f$ form a locally Lipschitz foliation by lines. Now we fix any integral curve $\gamma_0$, and assume that $\gamma_0^-$ representing $\mathcal{E}_-$ and $\gamma_0^+$ representing $\mathcal{E}_+$, here $\mathcal{E}_-, \mathcal{E}_+ \in M(\mathscr{E})$ may coincide. Here, and in the following, for a line $\gamma:\mathbb{R} \rightarrow M$, $\gamma^+$ and $\gamma^-$ are rays defined by $\gamma^+(t):=\gamma(t)$ and $\gamma^-(t):=\gamma(-t)$ respectively for all $t \geq 0$.

First we prove for any other integral curve $\gamma_1$, we must have $\gamma_1^-$ representing $\mathcal{E}_-$ and $\gamma_1^+$ representing $\mathcal{E}_+$.    Since $M$ is connected, as a manifold, it is path-connected [\cite{Ga1}, Theorem 4]. So there exists a smooth curve $\xi:[0,1] \rightarrow M$ connecting $\gamma_0(0)$ and $\gamma_1(0)$, i.e. $\xi(0)=\gamma_0(0)$ and $\xi(1)=\gamma_1(0)$. We denote the flow generated by $\nabla f$ by $\phi_t$. Now we will get  two facts:

1). For any $t \geq 0$, $\gamma_0^+(t)$ and $\gamma_1^+(t)$ are connected by the curve $\phi_t(\xi)$;

2). Since $f$ is a $C^1$ (in fact, $C^{1,1}$) solution to the Hamilton-Jacobi equation, for any compact subset $K$, there exists a $t_K >0$ such that for $t > t_K$, $\phi_t(\xi) \cap K = \emptyset$.

Combining facts 1) and 2), we obtain that  $\gamma_0^+$ and $\gamma_1^+$ represent the same end (In fact, $\gamma_0^+$ and $\gamma_1^+$ are strongly equivalent, which is  a stronger condition introduced by Hopf than representing the same end, for definition and details, see \cite{Ho}, [\cite{Gu}, 3.3]). Similarly, $\gamma^{-}_0$ and $\gamma^{-}_1$ represent the same end. 

Up to now, we have proved that for any integral curves  $\gamma$, $\gamma^-$ represents $\mathcal{E}_-$ and $\gamma^+$ represents $\mathcal{E}_+$. Based on this fact, we could continue the proof as following.

Since $\#(\mathscr{E}(M)) \geq 3$, there exist a sufficiently large compact subset $K$,  a ray $\zeta$ representing an end $\mathcal{E}$ different from $\mathcal{E}_+$ and $\mathcal{E}_-$, and a real number   $T>0$ such that $\zeta|[T, \infty)$ and $\gamma_0^+|[T, \infty)$ lie in different connected components of $M \setminus K$,  $\zeta|[T, \infty)$ and $\gamma_0^-|[T, \infty)$ lie in different connected components of $M \setminus K$. Denote the diameter of $K$ by $r$ ($r>0$, since $K$ cannot to be a single point set in our case). Choose $S$  large enough such that $d(\zeta(T+S), K) > r$. Considering the integrable curve  $\gamma^{\prime}: \mathbb{R} \rightarrow M$ of $\nabla f$ with $\gamma^{\prime}(0)=\zeta(T+S)$, since ${\gamma^{\prime}}^{+}$ represents $\mathcal{E}_+$ and ${\gamma^{\prime}}^{-}$ represents $\mathcal{E}_-$,  there exist two real numbers $t^+$ and $t^-$ such that:

$\bullet$ $\gamma^{\prime}(t^+) \in K$,   $\gamma^{\prime}|(t^+, \infty)$ and $\gamma_0^+|(T, \infty)$ lie in the same connected component of $M \setminus K$.

$\bullet$ $\gamma^{\prime}(t^-) \in K$,   $\gamma^{\prime}|(-\infty, t^-)$ and $\gamma_0^-|(T, \infty)$ lie in the same connected component of $M \setminus K$.

Since $\gamma^{\prime}$ are lines, we obtain $d(\gamma^{\prime}(t^+), \gamma^{\prime}(t^-))=|t^+-t^-| > 2r.$ On the other hand, the fact that $\gamma^{\prime}(t^+) \in K$ and $\gamma^{\prime}(t^-) \in K$ will imply that $d(\gamma^{\prime}(t^+), \gamma^{\prime}(t^-)) \leq r$. This contradiction proves  Proposition 6.6 and thus Theorem 4. 2) is proved.

\end{proof}

\begin{ack}
The author would like to thank L. Jin for helpful discussions.
\end{ack}

\end{document}